\RequirePackage{fix-cm}
\RequirePackage{multicol, multirow}

\documentclass[twocolumn]{article}

\usepackage[english]{babel}
\usepackage{fullpage}
\usepackage{natbib}

\normalfont
\usepackage[a4paper,citecolor=blue,linkcolor=blue,colorlinks=true]{hyperref}

\usepackage{amsmath}
\usepackage{amssymb}
\usepackage{amsfonts}
\usepackage{amsthm}
\usepackage{graphics}
\usepackage{graphicx}
\usepackage[all]{xy}
\usepackage[usenames]{color}
\usepackage[algo2e, ruled]{algorithm2e}
\usepackage{dsfont}
\usepackage{mathrsfs}

\bibpunct{(}{)}{,}{a}{,}{,}

\newtheorem{theorem}{Theorem}

\newtheorem{proposition}[theorem]{Proposition}

\theoremstyle{definition}
\newtheorem{definition}[theorem]{Definition}

\newcommand{\s}{\mathscr{S}}
\newcommand{\G}{\mathscr{G}}

\newcommand{\X}{\mathscr{X}}
\newcommand{\Y}{\mathscr{Y}}
\newcommand{\yobs}{y^{\text{obs}}}
\newcommand{\ivj}{i\stackrel{\G}{\sim} j}
\newcommand{\Z}{Z(\G,\beta)}

\newcommand{\indxixj}{\mathds{1}\{x_i=x_j\}}
\newcommand{\indxyi}{\mathds{1}\{x_i=y_i\}}
\newcommand{\nref}{n_{\text{REF}}}

\newcommand{\HPM}{\text{HPM}(\G,\alpha,\beta)}
\newcommand{\HPMf}{\text{HPM}(\G_4,\alpha,\beta)}
\newcommand{\HPMh}{\text{HPM}(\G_8,\alpha,\beta)}

\newcommand{\GGy}{\Gamma(\G,y)}

\newcommand{\dd}{\mathrm{d}}
\newcommand{\prob}{\mathbb{P}}

\graphicspath{{images/}}

\begin{document}

\title{\bf Adaptive ABC model choice and geometric summary statistics for hidden Gibbs random fields}



%



\date{Received: date / Accepted: date}
\author{
Julien Stoehr\footnote{I3M -- UMR CNRS 5149, 
Universit\'e Montpellier II, France},
\and Pierre Pudlo$^{\ast}$
\and Lionel Cucala$^{\ast}$ 
}

\date{\today}

\maketitle

\begin{abstract}
  Selecting between different dependency structures of hidden Markov
  random field can be very challenging, due to the intractable
  normalizing constant in the likelihood. We answer this question with
  approximate Bayesian computation (ABC) which provides a model choice
  method in the Bayesian paradigm. This comes after the work of
  \citet{grelaud09} who exhibited sufficient statistics on directly
  observed Gibbs random fields. But when the random field is latent,
  the sufficiency falls and we complement the set with geometric
  summary statistics. The general approach to construct these
  intuitive statistics relies on a clustering analysis of the sites
  based on the observed colors and plausible latent graphs. 
  The efficiency of ABC model choice based on these statistics is
  evaluated via a local error rate which may be of independent
  interest. As a byproduct we derived an ABC algorithm that adapts
  the dimension of the summary statistics to
  the dataset without
  distorting the model selection.

\end{abstract}

\section{Introduction}

Gibbs random fields are polymorphous statistical models, that are
useful to analyse different types of spatially correlated data, with a
wide range of applications, including image analysis \citep{hurn03},
disease mapping \citep{green02}, genetic analysis \citep{francois06}
among others.  The autobinomial model \citep{besag74} which
encompasses the Potts model, is used to describe the spatial
dependency of discrete random variables (\textit{e.g.}, shades of grey
or colors) on the vertices of an undirected graph (\textit{e.g.}, a
regular grid of pixels). See for example \citet{alfo08} and
\citet*{moores14} who performed image segmentation with the help of
the above modeling.  Despite their popularity, these models present
major difficulties from the point of view of either parameter
estimation \citep{frielcache, frielproc, everitt12} or model choice
\citep{grelaud09, frielevidence, cucala13}, due to an intractable
normalizing constant. Remark the exception of small latices on which
we can apply the recursive algorithm of \citet{reeves04, friel07} and
obtain a reliable approximation of the normalizing constant. However,
the complexity in time of the above algorithm grows exponentially and
is thus helpless on large lattices.

The present paper deals with the challenging problem of selecting a
dependency structure of an hidden Potts model in the Bayesian paradigm
and explores the opportunity of approximate Bayesian computation
\citep[ABC,][]{tavare97,pritchard99, surveyABC, baragatti14} to answer
the question. Up to our knowledge, this important question has not yet
been addressed in the Bayesian literature. Alternatively we could have
tried to set up a reversible jump Markov chain Monte Carlo, but
follows an important work for the statistician to adapt the general
scheme, as shown by \citet{caimo11, caimo2013} in the context of
exponential random graph models where the observed data is a graph.
\citet{cucala13} addressed the question of inferring the number of
latent colors with an ICL criterion but their complex algorithm cannot
be extended easily to choose the dependency structure.  Other
approximate methods have also been tackled in the literature such as
pseudo-likelihoods \citep{besag75}, mean field approximations
\citep{forbes03} but lacks theoretical support.

Approximate Bayesian computation (ABC) is a simulation based approach
that can addresses the model choice issue in the Bayesian paradigm.
The algorithm compares the observed data $\yobs$ with numerous
simulations $y$ through summary statistics $S(y)$ in order to supply a
Monte Carlo approximation of the posterior probabilities of each
model. The choice of such summary statistics presents major
difficulties that have been especially highlighted for model choice
\citep{robert11, didelot2011}. Beyond the seldom situations where
sufficient  statistics exist and are explicitly known
\citep[Gibbs random fields are surprising examples, see][]{grelaud09},
\citet{marin14} provide conditions which ensure the consistency of ABC
model choice. The present work has thus to answer the absence of available
sufficient statistics for hidden Potts fields as well as the difficulty (if not
the impossibility) to check the above theoretical conditions in practice.

Recent articles have proposed automatic schemes to construct theses
statistics (rarely from scratch but based on a large set of
candidates) for Bayesian parameter inference and are meticulously
reviewed by \citet{blum2013} who compare their performances in
concrete examples. But very few has been accomplished in the context
of ABC model choice apart from the work of \citet{prangle14}. The
statistics $S(y)$ reconstructed by \citet{prangle14} have good
theoritical properties (those are the posterior probabilities of the
models in competition) but are poorly approximated with a pilot ABC
run \citep{robert11}, which is also time consuming.

The paper is organized as follows: Section \ref{sec:ABC} presents ABC
model choice as a $k$-nearest neighbor classifier, and defines a local
error rate which is the first contribution of the paper. We also
provide an adaptive ABC algorithm based on the local error to select
automatically the dimension of the summary statistics. The second
contribution is the introduction of a general and intuitive approach
to produce geometric summary statistics for hidden Potts model in
Section~\ref{sec:HPM}. We end the paper with numerical results in that
framework.


\section{Local error rates and adpative ABC model choice}
\label{sec:ABC}
When dealing with models whose likelihood cannot be computed
analytically, Bayesian model choice becomes challenging since the
evidence of each model writes as the integral of the likelihood over
the prior distribution of the model parameter. ABC provides a method
to escape from the intractability problem and relies on many simulated
datasets from each model either to learn the model that fits the
observed data $\yobs$ or to approximate the posterior probabilities.
We refer the reader to reviews on ABC \citep{surveyABC, baragatti14}
to get a wider presentation and will focus here on the model choice
procedure.

\subsection{Background on Approximate Bayesian computation for model choice}
\label{subsec-abc}

Assume we are given $M$ stochastic models with respective paramater
spaces embedded into Euclidean spaces of various dimensions. The joint
Bayes\-ian distribution sets
\begin{itemize}
\item[\emph{(i)}] a prior on the model space, $\pi(1),\ldots,\pi(M)$,
\item[\emph{(ii)}] for each model, a prior on its parameter
  space, whose density with respect to a reference measure (often the
  Lebesgue measure of the Euclidean space) is $ \pi_m(\theta_m) $ and
\item[\emph{(iii)}] the likelihood of the data $y$ within each model,
  namely $f_m(y|\theta_m)$.
\end{itemize}
The evidence of model $m$ is then
defined as
\[
e(m,y):=\int f_m(y|\theta_m)\pi_m(\theta_m)\dd \theta_m
\]
and the posterior probability of model $m$ as 
\begin{equation}\label{eq:pp}
\pi(m|y)=\frac{\pi(m)e(m,y)}{\sum_{m'}\pi(m')e(m',y)}.
\end{equation}
When the goal of the Bayesian analysis is the selection of the model that best fits the observed data
$\yobs$,
it is performed through the maximum a posteriori (MAP) defined by 
\begin{equation}
\widehat{m}_\text{MAP}(\yobs)=\operatorname{arg\,max}_m \pi(m|\yobs).
\label{eq:bayes.classifier}
\end{equation}
The latter  can be seen as a classification problem predicting
the model number given the observation of $y$. From this standpoint,
$\widehat{m}_\text{MAP}$ is the Bayes classifier, well known to minimize
the 0-1 loss \citep{devroye:gyorfi:lugosi:1996}.
One might argue that $\widehat{m}_\text{MAP}$ is an estimator defined
as the mode of the posterior probabilities which form the density of
the posterior with respect to the counting measure. But the counting
measure, namely $\delta_{1}+\cdots+\delta_{M}$, is a canonical
reference measure, since it is invariant to any permutation of
$\{1,\ldots,M\}$ whereas no such canonical reference measure
(invariant to one-to-one transformation) exists on
compact subset of the real line. Thus \eqref{eq:bayes.classifier} does
not suffer from the drawbacks of posterior mode estimators
\citep{druilhet:marin:2007}. 

To approximate $\widehat{m}_\text{MAP}$, ABC starts by simulating
numerous triplets $(m,\theta_m,y)$ from the joint Bayesian
model. Afterwards, the algorithm mimics the Bayes classifier
\eqref{eq:bayes.classifier}: it approximates the posterior
probabilities by the frequency of each model number associated with
simulated $y$'s in a neighborhood of $\yobs$. If required, we can
eventually predict the best model with the most frequent model in the
neighborhood, or, in other words, take the final decision by plugging
in \eqref{eq:bayes.classifier} the approximations of the posterior
probabilities.

If directly applied, this first, naive algorithm faces the curse of
dimensionality, as simulated datas\-ets $y$ can be complex objects and
lie in a space of high dimension (\textit{e.g.}, numerical
images). Indeed, finding a simulated dataset in the vicinity of
$\yobs$ is almost impossible when the ambient dimension is high. The
ABC algorithm performs therefore a (non linear) projection of the
observed and simulated datasets onto some Euclidean space of
reasonable dimension via a function $S$, composed of summary
statistics. Moreover, due to obvious reasons regarding
computer memory, instead of keeping track of the whole simulated datasets, 
one commonly saves only the simulated vectors of
summary statistics, which leads to a table composed of iid replicates
$(m,\theta_m, S(y))$, often called the reference table in the ABC
literature, see Algorithm~\ref{algo:reftable}.

\begin{algorithm2e}
  \caption{Simulation of the ABC reference table}
  \label{algo:reftable}
  \KwOut{A reference table of size $\nref$ }
  \medskip
  
  \For{$j\leftarrow 1$ \KwTo $\nref$}{
    \textbf{draw} $m$ from the prior $\pi$\; 
    \textbf{draw} $\theta$ from the prior $\pi_m$\; 
    \textbf{draw} $y$ from the likelihood $f_m(\cdot\vert\theta)$\; 
    \textbf{compute} $S(y)$\; 
    \textbf{save} $(m_j, \theta_j, S(y_j))\leftarrow(m,\theta,S(y))$\;
  }
  \textbf{return} the table of $(m_j, \theta_j, S(y_j))$, $j=1,\ldots, \nref$
\end{algorithm2e}


From the standpoint of machine learning, the reference table serves as
a training database composed of iid replicates drawn from the distribution
of interest, namely the joint Bayesian model. The regression problem
of estimating the posterior probabilities or the classification
problem of predicting a model number are both solved with
nonparametric methods. The neighborhood of $\yobs$ is thus defined as
simulations whose distances to the observation measured in terms of
summary statistics, \textit{i.e.}, $\rho(S(y), S(\yobs))$, fall below
a threshold $\varepsilon$ commonly named the tolerance level. The
calibration of $\varepsilon$ is delicate, but had been partly
neglected in the papers dealing with ABC that first focused on
decreasing the total number of simulations via the recourse to Markov chain Monte Carlo  \citep{marjoram:etal:2003} or sequential Monte
Carlo methods  \citep{beaumont2009adaptive, delmoral:doucet:jasra:2009} 
whose common target is the joint Bayesian distribution
conditioned by $\rho(S(y), S(\yobs)) \le \varepsilon$ for a given
$\varepsilon$. By contrast, the simple setting we adopt here reveals
the calibration question. In accepting the machine learning viewpoint,
we can consider the ABC algorithm as a $k$-nearest neighbor (knn)
method, see \citet{biau13}; the calibration of $\varepsilon$ is thus
transformed into the calibration of $k$. The Algorithm we have to
calibrate is given in Algorithm~\ref{algo:ABC}.


\begin{algorithm2e}
  \caption{Uncalibrated ABC model choice }
  \label{algo:ABC}
  \KwOut{A sample of size $k$ distributed according to
    the ABC approximation of the posterior}

  \medskip
  \textbf{simulate} the reference table $\mathscr T$ according to Algorithm~\ref{algo:reftable}\;
  \textbf{sort} the replicates of $\mathscr T$ according to $\rho(S(y_j),S(\yobs))$\;
  \textbf{keep} the $k$ first replicates\;
  \textbf{return} the relative frequencies of each model among the $k$ first
  replicates and the most frequent model\;
\end{algorithm2e}


Before entering into the tuning of $k$, we highlight that the
projection via the summary statistics generates a difference with the
standard knn methods. Under mild conditions, knn are consistent
nonparametric methods. Consequently, as the size of the reference table tends
to infinity, the relative frequency of model $m$ returned by
Algorithm~\ref{algo:ABC} converges to 
\[
\pi(m | S(\yobs)).
\]
Unfortunately, when the summary statistics are not sufficient for the model
choice problem, \citet{didelot2011} and \citet{robert11} found that
the above probability can greatly differ from the genuine $\pi(m|\yobs)$.
Afterwards \citet{marin14} provide necessary
and sufficient conditions on $S(\cdot)$ for the consistency of the MAP
based on $\pi(m | S(\yobs))$ when the information included in the dataset 
$\yobs$ increases, \textit{i.e.} when the dimension of 
$\yobs$ tends to infinity. Consequently, the problem that ABC 
addresses with reliability is classification, and the mentioned
theoretical results requires a shift from the approximation of
posterior probabilities. Practically the frequencies returned by
Algorithm~\ref{algo:ABC} should solely be used to order the 
models with respect to their fits to $\yobs$ and construct a knn classifier
$\widehat{m}$ that predicts the model number.

It becomes therefore obvious that the calibration of $k$ should be
done by minimizing the misclassification error rate of the resulting
classifier $\widehat{m}$.  This indicator is the expected value of the
0-1 loss function, namely $\mathbf 1\{\widehat{m}(y)\neq m\}$, over a
random $(m,y)$ distributed according to the marginal (integrated in
$\theta_m$) of the joint Bayesian distribution, whose density in $(m,y)$ writes
\begin{equation} \label{eq:prior.predictive}
\pi(m) \int f_m(y|\theta_m)\pi_m(\theta_m) \dd\theta_m.
\end{equation}
Ingenious solutions have been already proposed and are now well
established to fullfil this minimization goal and bypass the
overfitting problem, based on cross-validation on the learning
database. But, for the sake of clarity, particularly in the following
sections, we decided to take advantage of the fact that ABC aims at
learning on simulated databases, and we will use a validation reference table,
simulated also with Algorithm~\ref{algo:reftable}, but independently of the
training reference table, to evaluate the misclassification rate with
the averaged number of  differences between the true model numbers $m_j$ and
the predicted values $\widehat{m}(y_j)$ by knn (\textit{i.e.} by ABC) on
the validation reference table.

\subsection{Local error rates}
\label{sub:local}

The misclassification rate $\tau$ of the knn classifier $\hat{m}$ at
the core of Algorithm~\ref{algo:ABC} provides consistant evidence of
its global accuracy. It supplies indeed a well-known support to
calibrate $k$ in Algorithm~\ref{algo:ABC}.
The purpose of ABC model choice methods though is the analyse of an
observed dataset $\yobs$ and this first indicator is irrelevant to
assess the accuracy of the classifier at this precise point of the
data space, since it is by nature a prior gauge. We propose here to
disintegrate this indicator, and to rely on conditional expected value
of the misclassification loss $\mathbf 1\{\widehat{m}(y)\neq m\}$ knowing
$y$ as an evaluation of the efficiency of the classifier at $y$. We
recall the following proposition whose proof is easy, but might help
clarifying matters when applied to the joint distribution
\eqref{eq:prior.predictive}.
\begin{proposition} \label{pro:disintegrate}
  Consider a classifier $\hat{m}$ that aims at predicting $m$ given $y$
  on data drawn from the joint distribution $f(m,y)$. Let $\tau$ be the
  misclassification rate of $\hat{m}$, defined by
  $\prob(\hat{m}(Y)\neq\mathscr M)$, where $(\mathscr M, Y)$ is a random
  pair with distribution $f$ under the probability measure $\prob$. Then,
  \emph{(i)} the expectation of the loss function is
  \[
  \tau  = \sum_m\int_y \mathbf 1\{\widehat{m}(y)\neq m\} f(m,y)\, \dd y.
  \]
  Additionally,  \emph{(ii)}, the conditional expectation knowing $y$, namely
  $\tau(y)=\prob\big(\hat{m}(Y) \neq \mathscr M\big| Y=y\big)$, is
  \begin{equation}\label{eq:raw.disintegrate}
  \tau(y) = \sum_m \mathbf 1\{\widehat{m}(y)\neq m\}
    f(m|y)
  \end{equation}
  and $\tau = \int_y f(y) \tau(y) \, \dd y$, where $f(y)$ denotes the
  marginal distribution of $f$ (integrated over $m$) and $f(m|y)=
  f(m,y)/f(y)$ the conditional probability of $m$ given
  $y$. Furthermore, we have
  \begin{equation}
    \label{eq:raw.disintegrate2}
    \tau(y) = 1 - f(\widehat{m}(y)\,|\,y).
  \end{equation}
\end{proposition}
The last result \eqref{eq:raw.disintegrate2} suggests that a
conditional expected value of the misclassification loss is a valuable
indicator of the error at $y$ since it is admitted that the posterior
probability of the predicted model reveals the accuracy of the
decision at $y$.  Nevertheless, the whole simulated datasets are not
saved into the ABC reference table but solely some numerical summaries
$S(y)$ per simulated dataset $y$, as explained above. Thus the
disintegration process of $\tau$ is practically limited to the
conditional expectation of the loss knowing some non one-to-one
function of $y$. Its definition becomes thereforemuch more subtle than the
basic \eqref{eq:raw.disintegrate}.
Actually, the ABC classifier can be trained on a subset $S_1(y)$ of
the summaries $S(y)$ saved in the training reference table, or on some
deterministic function (we still write $S_1(y)$) of $S(y)$ that 
reduces the dimension, such as the projection on the LDA axes proposed
by \citet{estoup:etal:2012}.  To highlight this fact, the ABC
classifier is denoted by $\widehat{m}(S_1(y))$ in what follows.  It is
worth noting here that the above setting encompasses any dimension
reduction technique presented in the review of \citet{blum2013},
though the review in oriented on parameter inference.  Furthermore
we might want to disintegrate the misclassification rate with respect
to another projection $S_2(y)$ of the simulated data that can
or cannot be related to the summaries $S_1(y)$ used to train the
ABC classifier, albeit $S_2(y)$ is also limited to be a
deterministic function of $S(y)$. This yields the following.
\begin{definition}\label{def:local}
  The \textit{local error rate} of the $\widehat{m}(S_1(y))$ classifier with
  respect to $S_2(y)$ is 
  \[
  \tau_{S_1}(S_2(y)):=\prob\big(\widehat{m}(S_1(Y)) \neq\mathscr M\,\big|\, S_2(Y)=S_2(y)\big),
  \]
  where $(\mathscr{M}, Y)$ is a random variable with distribution
  given in \eqref{eq:prior.predictive}.
\end{definition}

The purpose of the local misclassification rate in the present paper
is twofold and requires to play with the distinction between $S_1$ and
$S_2$, as the last part will show on numerical examples. The first
goal is the construction of a prospective tool that aims at checking
whether a new statistic $S'(y)$ carries additional information
regarding the model choice, beyond a first set of statistics
$S_1(y)$. In the latter case, it can be useful to localize the
misclassification error of $\widehat{m}(S_1(y))$ with respect to the
concatenated vector $S_2(y)=(S_1(y), S'(y))$.  Indeed, this local
error rate can reveal concentrated areas of the data space,
characterized in terms of $S_2(y)$, in which the local error rate
rises above $(M-1)/M$, the averaged (local) amount of errors of the
random classifier among $M$ models, so as to approach $1$. The
interpretation of the phenomenon is as follows: errors committed by
$\widehat{m}(S_1(y))$, that are mostly spread on the $S_1(y)$-space,
might gather in particular areas of subspaces of the support of
$S_2(y)=(S_1(y), S'(y))$.
This peculiarity is due to
the dimension reduction of the summary statistics in ABC before the
training of the classifier and represents a concrete proof of the
difficulty of ABC model choice already raised by \citet{didelot2011}
and \citet{robert11}.

The second goal of the local error rate given in
Definition~\ref{def:local} is the evaluation of the confidence we may
concede in the model predicted at $\yobs$ by $\widehat{m}(S_1(y))$, in
which case we set $S_2(y)=S_1(y)$.
And, when both sets of summaries agree, the
results of Proposition~\ref{pro:disintegrate} extend to
\begin{align} 
\tau_{S_1}(S_1(y)) &= \sum_m \pi(m|S_1(y))\mathbf
1\{\widehat{m}(S_1(y))=m\} \notag
\\
&= 1 - \pi(\widehat{m}(S_1(y)) \, |\, S_1(y)).
\label{eq:local.error.agree}
\end{align}

Besides the local error rate we propose in Definition~\ref{def:local}
is an upper bound of the Bayes classifier
if we admit the loss of information committed by replacing $y$ with the summaries.
\begin{proposition} \label{pro:compare.knn.Bayes}
  Consider any classifier $\widehat{m}(S_1(y))$. The local error rate
  of this classifier satisfies
  \begin{align}
    \tau_{S_1}(s_2) &= \prob\left(\widehat{m}(S_1(Y)) \neq \mathscr
      M \,\middle|\, S_2(Y)=s_2 \right) \notag
    \\
    &\ge  
    \prob\left(\widehat{m}_\text{\rm MAP}(Y) \neq \mathscr
      M \,\middle|\, S_2(Y)=s_2 \right),\label{eq:bound}
  \end{align}
  where $\widehat{m}_\text{\rm MAP}$ is the Bayes classifier defined in
  \eqref{eq:bayes.classifier} and $s_2$ any value in the support of $S_2(Y)$. 
  Consequently,
  \begin{equation}
    \label{eq:1}
    \prob\left(\widehat{m}(S_1(Y)) \neq \mathscr
      M \right) \ge \prob\left(\widehat{m}_\text{\rm MAP}(Y) \neq \mathscr
      M \right).
  \end{equation}
\end{proposition}
\begin{proof}
  Proposition~\ref{pro:disintegrate}, in particular
  \eqref{eq:raw.disintegrate2}, implies that
  $\widehat{m}_\text{MAP}(y)$ is the ideal classifier that minimizes the
  conditional $0$-$1$ loss knowing $y$. Hence, we have
  \begin{multline*}
   \prob\left(\widehat{m}(S_1(Y)) \neq \mathscr
      M \,\middle|\, Y=y \right) 
    \\
    \ge  
    \prob\left(\widehat{m}_\text{\rm MAP}(Y) \neq \mathscr
      M \,\middle|\, Y=y \right).
  \end{multline*}  
  Integrating the above with respect to the distribution of $Y$
  knowing $S_2(Y)$ leads to \eqref{eq:bound}, and a last integral to
  \eqref{eq:1}.
\end{proof}

Proposition~\ref{pro:compare.knn.Bayes} shows that the introduction of
new summary statistics cannot distort the model selection insofar as
the risk of the resulting classifier cannot decrease below the risk of
the Bayes classifier $\widehat{m}_\text{MAP}$. We give here a last
flavor of the results of \citet{marin14} and mention that, if
$S_1(y)=S_2(y)=S(y)$ and if the classifiers are perfect
(\textit{i.e.}, trained on infinite reference tables), we can rephrase
part of their results as providing mild conditions on $S$ under which
the local error $\tau_S(S(y))$ tends to $0$ when the size of the
dataset $y$ tends to infinity.

\subsection{Estimation algorithm of the local error rates}
\label{sub:estimate.local.error}

The numerical estimation of the local error rate $\tau_{S_1}(S_2(y))$,
as a surface depending on $S_2(y)$, is therefore paramount to assess
the difficulty of the classification problem at any $s_2=S_S(y)$, and
the local accuracy of the classifier. Naturally, when $S_1(y)=S_2(y)$
for all $y$, the local error can be evaluated at $S_2(\yobs)$ by
plugging in \eqref{eq:local.error.agree} the ABC estimates of the posterior probabilities (the
relative frequencies of each model among the particles returned by
Algorithm~\ref{algo:ABC})  as
substitute for $\pi(m|S(\yobs))$. This estimation procedure is
restricted to the above mentioned case where the set of statistics used
to localize the error rate agrees with the set of statistics used to
train the classifier. Moreover, the approximation of the posterior
probabilities returned by Algorithm~\ref{algo:ABC}, \textit{i.e.}, a
knn method, might not be trustworthy: the calibration of $k$ performed
by minimizing the prior error rate $\tau$ does not provide any
certainty on the estimated posterior probabilities beyond a ranking of
these probabilities that yields the best classifier in terms of
misclassification. In other words, the knn method calibrated to answer
the classification problem of discriminating among models does not
produce a reliable answer to the regression problem of estimating
posterior probabilities. Certainly, the value of $k$ must be increased
to face this second kind of issue, at the price of a larger bias that
might even swap the model ranking (otherwise, the empirical prior
error rate would not depend on $k$, see the numerical result section).

For all these reasons, we propose here an alternative estimate of the
local error. The core idea of our proposal is the recourse to a
nonparametric method to estimate conditional expected values based on
the calls to the classifier $\widehat{m}$ on a validation reference
table, already simulated to estimate the global error rate
$\tau$. Nadaraya-Watson kernel estimators of the conditional
expected values 
\begin{multline}
\tau_{S_1}(S_2(y))=  \\
\mathbb{E}\left(\mathbf 1 \{\widehat{m}(S_1(Y)) \neq
\mathscr M\}\, \middle|\, S_2(Y)=S_2(y) \right) \label{eq:cond.exp}
\end{multline}
rely explicitly on the regularity of this indicator, as a function of
$s_2=S_2(y)$, which contrasts with the ABC plug-in estimate described
above. We thus hope improvements in the accuracy of error estimate
and a more reliable approximation of the whole function
$\tau_{S_1}(S_2(y))$. Additionally, we are not limited anymore to
the special case where $S_1(y)=S_2(y)$ for all $y$.  It is worth
stressing here that the bandwidth of the kernels must be calibrated by
minimizing the $L^2$-loss, since the target is a conditional expected
value.

\begin{algorithm2e}
\caption{Estimation of $\tau_{S_1}(S_2(y))$ given an classifier $\widehat{m}(S_1(y))$
on a validation or test reference table}
\label{algo:estim.local}
\KwIn{A validation or test reference table and a classifier $\widehat{m}(S_1(y))$ fitted with a first reference table}
\KwOut{Estimations of \eqref{eq:cond.exp} at each point of the second
reference table}

\medskip
\For{\rm\textbf{each} $(m_j,y_j)$ in the test table}{
  \textbf{compute} $\delta_j=\{\widehat{m}(S_1(y_j)) \neq m_j\}$\;
}

\textbf{calibrate} the bandwidth $\mathbf{h}$ of the Nadaraya-Watson estimator
predicting $\delta_j$ knowing $S_2(y_j)$ via cross-validation on the
test table\;

\For{\rm\textbf{each} $(m_j,y_j)$ in the test table}{
  \textbf{evaluate} the Nadaraya-Watson estimator with bandwidth
  $\mathbf{h}$ at $S_2(y_j)$\;
}
\end{algorithm2e}


Practically, this leads to Algorithm~\ref{algo:estim.local} which
requires a \textit{validation or test reference table} independent of
the training database that constitutes the ABC reference table. We can
bypass the requirement by resorting to cross validation methods, as
for the computation of the global prior misclassification rate
$\tau$. But the ensued algorithm is complex and it induces more calls
to the classifier (consider, e.g,. a ten-fold cross validation
algorithm computed on more than one random grouping of the reference
table) than the basic Algorithm~\ref{algo:estim.local}, whereas the
training database can always be supplemented by a validation database
since ABC, by its very nature, is a learning problem on simulated
databases.  Moreover, to display the whole surface
$\tau_{S_1}(S_2(y))$, we can interpolate values of the local error
between points $S_2(y)$ of the second reference table with the help of
a Kriging algorithm. We performed numerical experiments (not detailed
here) concluding that the resort to a Kriging algorithm provides
results comparable to the evaluation of Nadaraya-Watson estimator at
any point of the support of $S_2(y)$, and can reduce computation
times.


\subsection{Adaptive ABC}
\label{sub:adaptive}
The local error rate can also represent a valuable way to adjust the summary statistics to the data point $y$ and to build an
adaptive ABC algorithm achieving a local trade off that increases the
dimension of the summary statistics at $y$ only when the additional
coordinates add information regarding the classification problem.
Assume that we have at our disposal a collection of ABC classifiers,
$\widehat{m}_\lambda(y):=\widehat{m}_\lambda(S_\lambda(y))$,
$\lambda=1,\ldots,\Lambda$, trained on various projections of $y$, namely the
$S_\lambda(y)$'s, and that all these vectors, sorted with respect to
their dimension, depend only on the summary statistics 
registered in the reference tables. Sometimes low dimensional
statistics may suffice for the classification (of models) at $y$,
whereas other times we may need to examine statistics of larger
dimension.  The local adaptation of the classifier is
accomplished through the disintegration of the misclassification rates
of the initial classifiers with respect to a common statistic
$S_0(y)$. Denoting $\tau_{\lambda}(S_0(y))$ the local error rate of
$\widehat{m}_\lambda(y)$ knowing $S_0(y)$, this reasoning yields the
adaptive classifier defined by
\begin{multline} \label{eq:adaptive}
\widetilde{m}(S(y)) := \widehat{m}_{\,\widehat{\lambda}(y)}(y),
\\ \text{where }
\widehat{\lambda}(y) := \operatorname{arg\,min}_{\lambda=1,\ldots, \Lambda}
\tau_{\lambda}(S_0(y)). 
\end{multline}
This last classifier attempts to avoid bearing the cost of the
potential curse of dimensionality from which all knn classifiers
suffer and can help reduce the error of the initial classifiers,
although the error of the ideal classifier \eqref{eq:bayes.classifier}
remains an absolute lower bound, see
Proposition~\ref{pro:compare.knn.Bayes}. From a different perspective,
\eqref{eq:adaptive} represents a way to tune the similarity
$\rho(S(y),S(\yobs))$ of Algorithm~\ref{algo:ABC} that locally
includes or excludes components of $S(y)$ to assess the proximity
between $S(y)$ and $S(\yobs)$.  Practically, we rely on the following
algorithm to produce the adaptive classifier, that requires a
validation reference table independent of the reference table used to
fit the initial classifiers.

\begin{algorithm2e}
\caption{Adaptive ABC model choice}
\label{algo:adaptive.ABC}
\KwIn{A collection of classifiers $\widehat{m}_\lambda(y)$,
  $\lambda=1,\ldots,\Lambda$ and a validation  reference table}
\KwOut{An adaptive classifier $\widetilde{m}(y)$}

\medskip

\For{\rm\textbf{each} $\lambda\in\{1,\ldots,\Lambda\}$}{
  \textbf{estimate} the local error of $\widehat{m}_\lambda(y)$
  knowing $S_0(y)$ with the help of Algorithm~\ref{algo:estim.local}\;
}

\textbf{return} the adaptive classifier $\widetilde{m}$ as a function
computing \eqref{eq:adaptive}\;
\end{algorithm2e}


The local error surface estimated within the loop of
Algorithm~\ref{algo:adaptive.ABC} must contrast the errors of the
collection of classifiers. Our advice is thus to build a projection
$S_0(y)$ of the summaries $S(y)$ registered in the reference tables as
follow. Add to the validation reference table a qualitative trait
which groups the replicates of the table according to their
differences between the predicted numbers by the initial classifiers and
the model numbers $m_j$ registered in the database. For instance,
when the collection is composed of $\Lambda=2$ classifiers, the
qualitative trait takes three values: value $0$ when both classifiers
$\widehat{m}_\lambda(y_j)$ agree (whatever the value of
$\widehat{m}_j$),  value $1$ when the first classifier only returns the correct
number, \textit{i.e.}, $\widehat{m}_1(y_j)=m_j\neq
\widehat{m}_2(y_j)$, and value $2$ when the second classifier only
returns the correct number, \textit{i.e.}, $\widehat{m}_1(y_j)\neq m_j =
\widehat{m}_2(y_j)$. The axes of the linear discriminant analysis
(LDA) predicting the qualitative trait knowing $S(y)$ provide
a projection $S_0(y)$ which contrasts the errors of the initial collection
of classifiers.

Finally it is important to note that the local error rates are
evaluated in Algorithm~\ref{algo:adaptive.ABC} with the help of a
validation reference table. Therefore, a reliable estimation of the
accuracy of the adaptive classifier cannot be based on the same
validation database because of the optimism bias of the training
error.  Evaluating the accuracy requires the simulation of a
\textit{test reference table} independently of the two first databases
used to train and adapt the predictor, as is usually performed in the
machine learning community.


\section{Hidden random fields}
\label{sec:HPM}

Our primary intent with the ABC methodology exposed in
Section~\ref{sec:ABC} was the study of new summary statistics to
discriminate between hidden random fields models. The following
materials numerically illustrate how ABC can choose the
dependency structure of latent Potts models among two possible
neighborhood systems, both described with undirected graphs, whilst
highlighting the generality of the approach.

\subsection{Hidden Potts model}
\label{sub:Potts}
This numerical part of the paper focuses on hidden Potts models, that
are representative of the general level of difficulty while at the
same time being widely used in practice \citep[see for
example][]{hurn03,alfo08,francois06,moores14}.  we recall that the
latent random field $x$ is a family of random variables $x_i$ indexed
by a finite set $\s$, whose elements are called sites, and taking
values in a finite state space $\X:=\{0,\ldots,K-1\}$, interpreted as
colors. When modeling a digital image, the sites are lying on a
regular 2D-grid of pixels, and their dependency is given by an
undirected graph $\G$ which defines an adjacency relationship on the
set of sites $\s$: by definition, both sites $i$ and $j$ are adjacent
if and only if the graph $\G$ includes an edge that links directly $i$
and $j$. A Potts model sets a probability distribution on $x$,
parametrized by a scalar $\beta$ that adjusts the level of dependency
between adjacent sites. The latter class of models differs from the
auto-models of \citet{besag74}, that allow variations on the level of
dependencies between edges and introduce potential anisotropy on the
graph. But the difficulty of all these models arises from the
intractable normalizing constant, called the partition function, as
illustrated in the distribution of Potts models defined by
\[
\pi(x\vert\G, \beta)=\frac{1}{\Z}\exp\left(\beta\sum_{\ivj}\indxixj\right).
\]
The above sum $\ivj$ ranges the set of edges of the graph $\G$ and the normalizing constant $\Z$ 
writes as 
\begin{equation}
\Z=\sum_{x\in\X}
\exp\left(\beta\sum_{\ivj}\indxixj\right),
\label{eq-Z}
\end{equation}
namely a summation over the numerous possible realizations of the
random field $x$, that cannot be computed directly (except for small
grids and small number of colors $K$). In the statistical physic
literature, $\beta$ is interpreted as the inverse of a temperature,
and when the temperature drops below a fixed threshold, values $x_i$
of a typical realization of the field are almost all equal (due to
important dependency between all sites). These peculiarities of Potts
models are called phase transitions.

\begin{figure}[tb]
\centering
\begin{minipage}[t]{7cm}
\centering
\includegraphics [height = 3cm, width = 3cm]{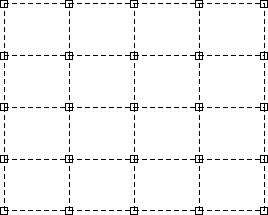}\\
(a)
\end{minipage}
\begin{minipage}[t]{7cm}
\centering
\includegraphics [height = 3cm, width = 3cm]{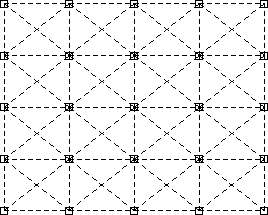}\\
(b)
\end{minipage}
\caption{Neighborhood graphs $\G$ of hidden Potts model. (a) The four closest neighbour graph $\G_4$ defining model $\HPMf$. (b) The eight closest neighbour graph $\G_8$ defining model $\HPMh$}
\label{fig:neigh}
\end{figure}

In hidden Markov random fields, the latent process is observed indirectly
through another field; this permits the modeling of a noise that may be
encountered in many concrete situations. Precisely, given the
realization $x$ of the latent field, the observation $y$ is a family
of random variables indexed by the set of sites, and taking values in
a set $\Y$, \textit{i.e.}, $y=\left(y_i; {i\in\s}\right)$, and are
commonly assumed as independent draws that form a noisy version of the
hidden fields. Consequently, we set the conditional distribution of
$y$ knowing $x$ as the product $\pi(y\vert x, \alpha)=\prod_{i\in\s}
P(y_i\vert x_i,\alpha)$, where $P$ is the marginal noise distribution
parametrized by some scalar $\alpha$. Hence the likelihood of the
hidden Potts model with parameter $\beta$ on the graph $\G$  and noise
distribution $P_\alpha$, denoted $\HPM$, is given by 
\[
f(y\vert\alpha ,\beta, \G) =\sum_{x\in\mathscr{X}}\pi\left(x\vert \G ,\beta\right)\pi_\alpha(y\vert x)
\]
and faces a double intractable issue as neither the likelihood of the
latent field, nor the above sum can be computed directly: the
cardinality of the range of the sum is of combinatorial
complexity. The following numerical experiments are based on two
classes of noises, producing either observations in $\{0,1,\ldots,
K-1\}$, the set of latent colors, or continuous observations that take
values in $\mathbb R$.

The common point of our examples is to select the hidden Gibbs model
that better fits a given $\yobs$ composed of $N=100\times 100$ pixels
within different neighborhood systems represented as undirected graphs
$\G$. We considered the two widely used adjacency structures in our
simulations, namely the graph $\G_4$ (respectively $\G_8$) in which
the neighborhood of a site is composed of the four (respectively
eight) closest sites on the two-dimensional lattice, except on the
boundaries of the lattice, see Fig.~\ref{fig:neigh}.  The prior
probabilities of both models were set to $1/2$ in all experiments.
The Bayesian analysis of the model choice question adds another
integral beyond the two above mentioned sums that cannot be calculated
explicitly or numerically either and the problem we illustrate are
said triple intractable. Up to our knowledge the choice of the latent
neighborhood structure has never been seriously tackled in the
Bayesian literature. We mentioned here the mean field approximation of
\citet{forbes03} whose software can estimate paramaters of such
models, and compare models fitness via a BIC criterion. But the first
results we try to obtain with this tool were worse than with ABC,
except for very low values of $\beta$. This means that either we did
not manage to run the software properly, or that the mean field
approximation is not appropriate to discriminate between neighborhood
structures.  The detailed settings of our three experiments are as
follows.

\paragraph{First experiment.} We considered Potts models with $K=2$
colors and a noise process that switches each pixel independently with
probability \[\exp(-\alpha)/(\exp(\alpha)+\exp(-\alpha)),\] following
the proposal of \citet{everitt12}. The prior on $\alpha$ was
uniform over $(0.42; 2.3)$, where the bounds of the interval were
determined to switch a pixel with a probability less than
$30\%$. Regarding the dependency parameter $\beta$, we set prior
distributions below the phase transition which occurs at different
levels depending on the neighborhood structure. Precisely we used  a uniform
distribution over $(0;1)$ when the adjacency is given by $\G_4$ and a
uniform distribution over $(0; 0.35)$ with $\G_8$.

\paragraph{Second experiment.} We increased the number of colors in
the Potts models and set $K=16$. Likewise, we set a noise
that changes the color of each pixel with a given probability
parametrized by $\alpha$, and conditionally
on a change at site $i$, we rely on the least favorable distribution,
which is a uniform draw within all colors except the latent one. To
extend the parametrization of \citet{everitt12}, the marginal 
distribution of the noise is defined by
\[
P_\alpha(y_i\vert x_i) = \frac{\exp \Big\{\alpha\big(2\indxyi -
  1\big)\Big\} }{\exp(\alpha) + (K-1)\exp(-\alpha)}
\]
and a uniform prior on $\alpha$ over the interval $(1.78;4.8)$ ensures
that the probability of changing a pixel with the noise process is at
most $30\%$. The uniform prior on the Potts parameter $\beta$ was also
tuned to stay below the phase transition. Hence $\beta$
ranges the interval $(0;2.4)$ with a $\G_4$ structure and
the interval $(0; 1)$ with a $\G_8$ structure.

\paragraph{Third experiment.} We introduced a homoscedastic Gaussian
noise whose marginal distribution is characterized by
\[
y_i~\vert ~ x_i =c \sim \mathcal{N}(c,\sigma^2) \quad c\in \{0;1\}
\]
over bicolor Potts models. And both prior distributions on parameter
$\beta$ are similar to the ones on the latent fields of the first
experiment.  The standard deviation $\sigma=0.39$ was set so that the
probability of a wrong prediction of the latent color with a marginal
MAP rule on the Gaussian model is about $15\%$.

\subsection{Geometric summary statistics}
\label{sub:geometry}

\begin{figure*}[t]
\centering
\begin{minipage}[t]{.45\textwidth}
\centering
\includegraphics[height=11em, width=11em]{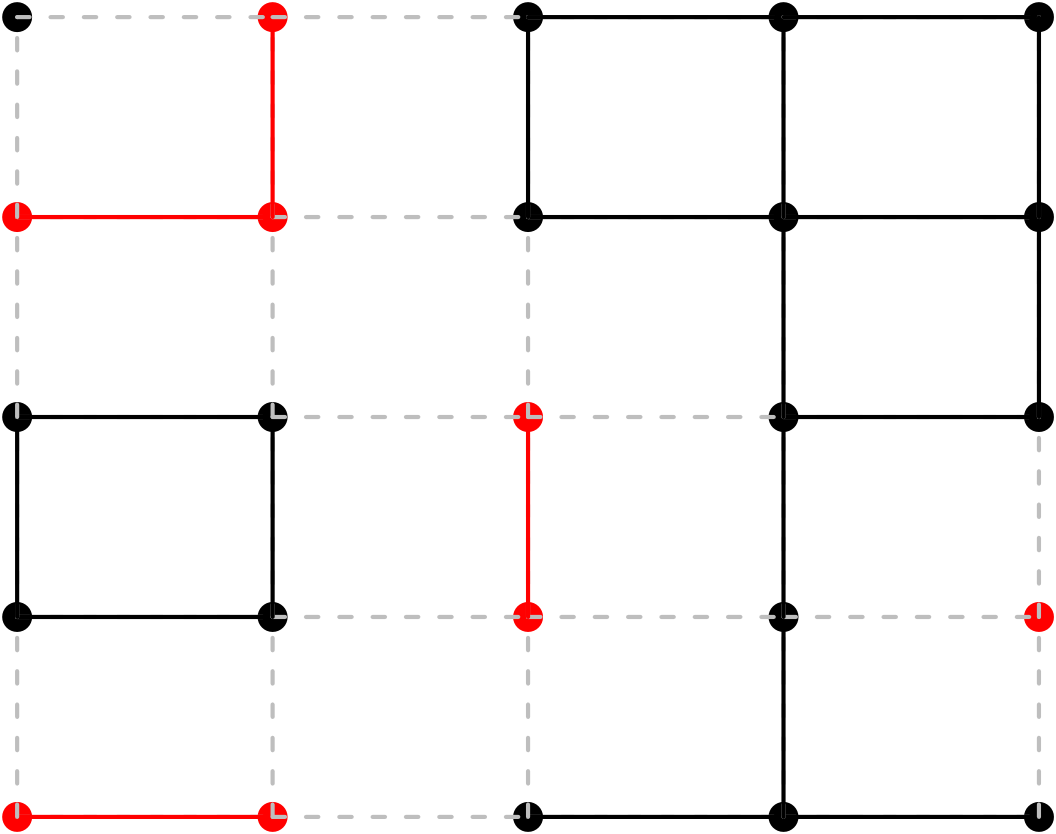}\\
$\Gamma(\G_4,y)$
\end{minipage} 
\begin{minipage}[t]{.45\textwidth}
\centering
\includegraphics[height=11em, width=11em]{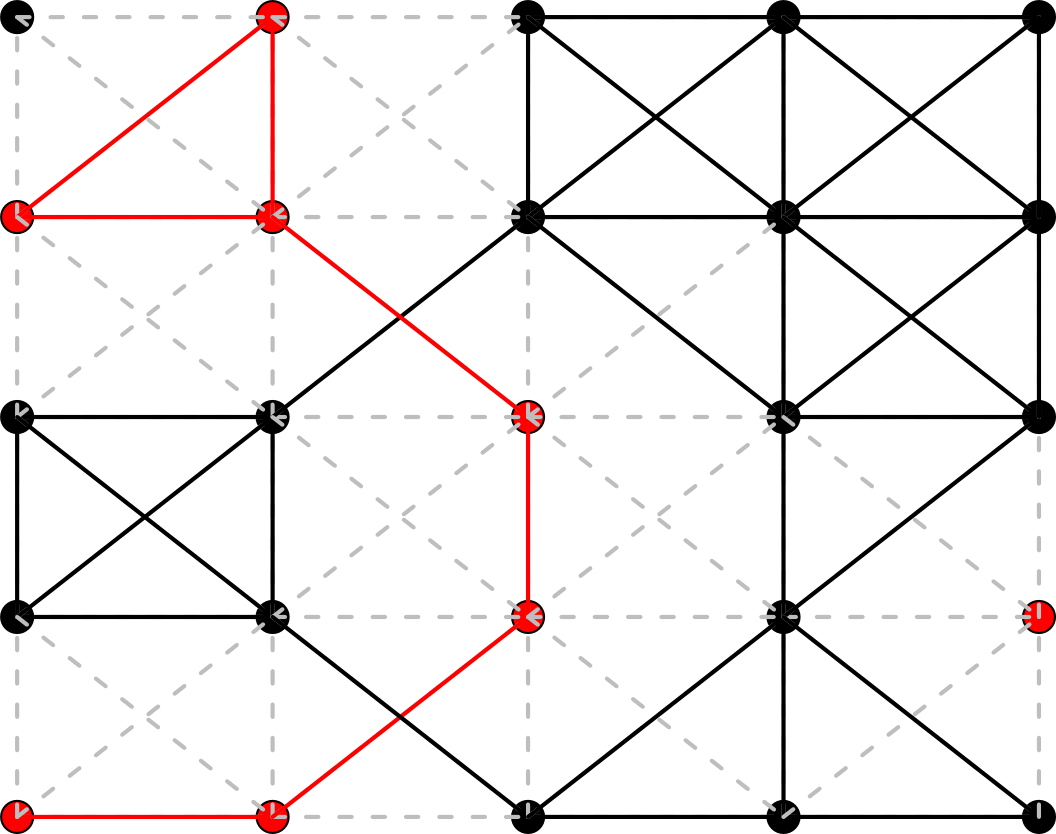}\\
$\Gamma(\G_8,y)$
\end{minipage}
\caption{The induced graph $\Gamma(\G_4,y)$ and $\Gamma(\G_8,y)$ on a given bicolor image $y$ of size $5\times 5$. The six summary statistics on $y$ are thus $R(\G_4,y) = 22$, $T(\G_4,y) = 7$, $U(\G_4,y)=12$, $R(\G_8,y) = 39$, $T(\G_8,y) = 4$ and $U(\G_8,y)=16$}
\label{fig:statgeom}
\end{figure*}

Performing a Bayesian model choice via ABC algorithms
requires summary statistics that capture the relevant information from
the observation $\yobs$ to discriminate among the competing
models. When the observation is noise-free, \citet{grelaud09} noted
that the joint distribution resulting from the Bayesian modeling
falls into the exponential family, and they obtained
consecutively a small set of summary statistics, depending on the
collection of considered models, that were sufficient. In front of
noise, the situation differs substantially as the joint distribution
lies now outside the exponential family, and the above mentioned
statistics are not sufficient anymore, whence the urge to bring forward
other concrete and workable statistics. 
The general approach we developed reveals geometric features of a
discrete field $y$ via the recourse to colored graphs attached to $y$
and their connected components. Consider an undirected graph $\G$
whose set of vertices coincides with $\s$, the set of sites of $y$.
\begin{definition}
  \label{def:induced.graph} 
  The graph induced by $\G$ on the field $y$, denoted $\GGy$, is the
  undirected graph whose set of edges gathers the edges of $\G$
  between sites of $y$ that share the same color, \textit{i.e.},
  \[
  i \stackrel{\Gamma(\G, y)}{\sim} j \quad \iff \quad
  i \stackrel{\G}{\sim} j \text{ and } y_i=y_j.
  \]
\end{definition}
We believe that the connected components of such induced graphs
capture major parts of the geometry of $y$.  Recall that a connected
component of an undirected graph $\Gamma$ is a subgraph of $\Gamma$ in
which any two vertices are connected to each other by a path, and
which is connected to no other vertices of $\Gamma$. And the connected
components form a partition of the vertices. Since ABC relies on the
computation of the summary statistics on many simulated datasets, it
is also worth noting that the connected components can be computed
efficiently with the help of famous graph algorithms in linear time
based on a breadth-first search or depth-first search over the
graph. The empirical distribution of the sizes of the connected
components represents an important source of geometric informations,
but cannot be used as a statistic in ABC because of the curse of
dimensionality. The definition of a low dimensional summary statistic
derived from these connect components should be guided by the
intuition on the model choice we face.

Our numerical experiments discriminate between a $\G_4$- and a
$\G_8$-neighborhood structure and we considered two induced graphs on
each simulated $y$, namely $\Gamma(\G_4, y)$ and $\Gamma(\G_8,
y)$. Remark that the two-dimensional statistics proposed by
\citet{grelaud09}, which are sufficient in the noise-free context, are
the total numbers of edges in both induced graphs. After very few
trials without success, we fixed ourselves on four additional summary
statistics, namely the size of the largest component of each induced
graph, as well as the total number of connect components in each
graph. See Fig.~\ref{fig:statgeom} for an example on a bicolor picture
$y$. To fix the notations, for any induced graph $\Gamma(\G, y)$, we
define
\begin{itemize}
\item $R(\G,y)$ as the total number of edges in $\Gamma(\G, y)$,
\item $T(\G,y)$ as the number of connected components in $\Gamma(\G,
  y)$ and
\item $U(\G,y)$ as the size of the largest connected component of $\Gamma(\G, y)$.
\end{itemize}
And to sum up the above, the set of summary statistics that where
registered in the reference tables for each simulated field $y$ is
\begin{multline*}
S(y) = \Big(R(\G_4,y); R(\G_8,y); T(\G_4,y);
\\
 T(\G_8,y); U(\G_4,y); U(\G_8,y)\Big)
\end{multline*}
in the first and second experiments.

In the third experiment, the observed field $y$ takes values in
$\mathbb R$ and we cannot apply directly the approach based on induced
graphes because no two pixels share the same color. All of the above
statistics are meaningless, including the statistics $R(\G,y)$ used by
\citet{grelaud09} in the noise-free case. We rely on a
quantization preprocessing performed via a kmeans algorithm on
the observed colors that forgets the spatial structure of the
field. The algorithm was tuned to uncover the same number of groups of
colors as the number of latent colors, namely $K=2$. If $q_2(y)$
denotes the resulting field,  the set of summary statistics becomes

\begin{multline*}
  S(y) = \Big(R\big(\G_4,q_2(y)\big); R\big(\G_8,q_2(y)\big);
  T\big(\G_4,q_2(y)\big);
  \\
  T\big(\G_8,q_2(y)\big);
  U\big(\G_4,q_2(y)\big); U\big(\G_8,q_2(y)\big)\Big).
\end{multline*}
We have assumed here that the number of latent colors is known to keep
the same purpose of selecting the correct neighborhood
structure. Indeed \citet{cucala13} have already proposed a (complex)
Bayesian method to infer the appropriate number of hidden colors. But
more generally, we can add statistics based on various quantizations
$q_k(y)$ of $y$ with $k$ groups.


\subsection{Numerical results}
\label{sub:numerical}

\begin{table}[b]
  \centering
  \caption{Evaluation of the prior error rate on a test reference
    table of size $30,000$ in the first experiment.}
  \label{tab:2colors}

  \textbf{Prior error rates}

  \begin{tabular}{ccc}
    \hline
    \textbf{Train size} & $\bf 5,000$ & $\bf 100,000$ 
    \\
    \hline
    2D statistics & $8.8 \%$  & $7.9 \%$
    \\
    4D statistics & $6.5 \%$ & $6.1 \%$
    \\
    6D statistics & $7.1 \%$ & $7.1 \%$ 
    \\
    Adaptive ABC & $6.2 \%$ & $5.5 \%$
    \\
    \hline
  \end{tabular}

\end{table}

\begin{figure*}
  \centering
    \includegraphics[width=.48\textwidth]{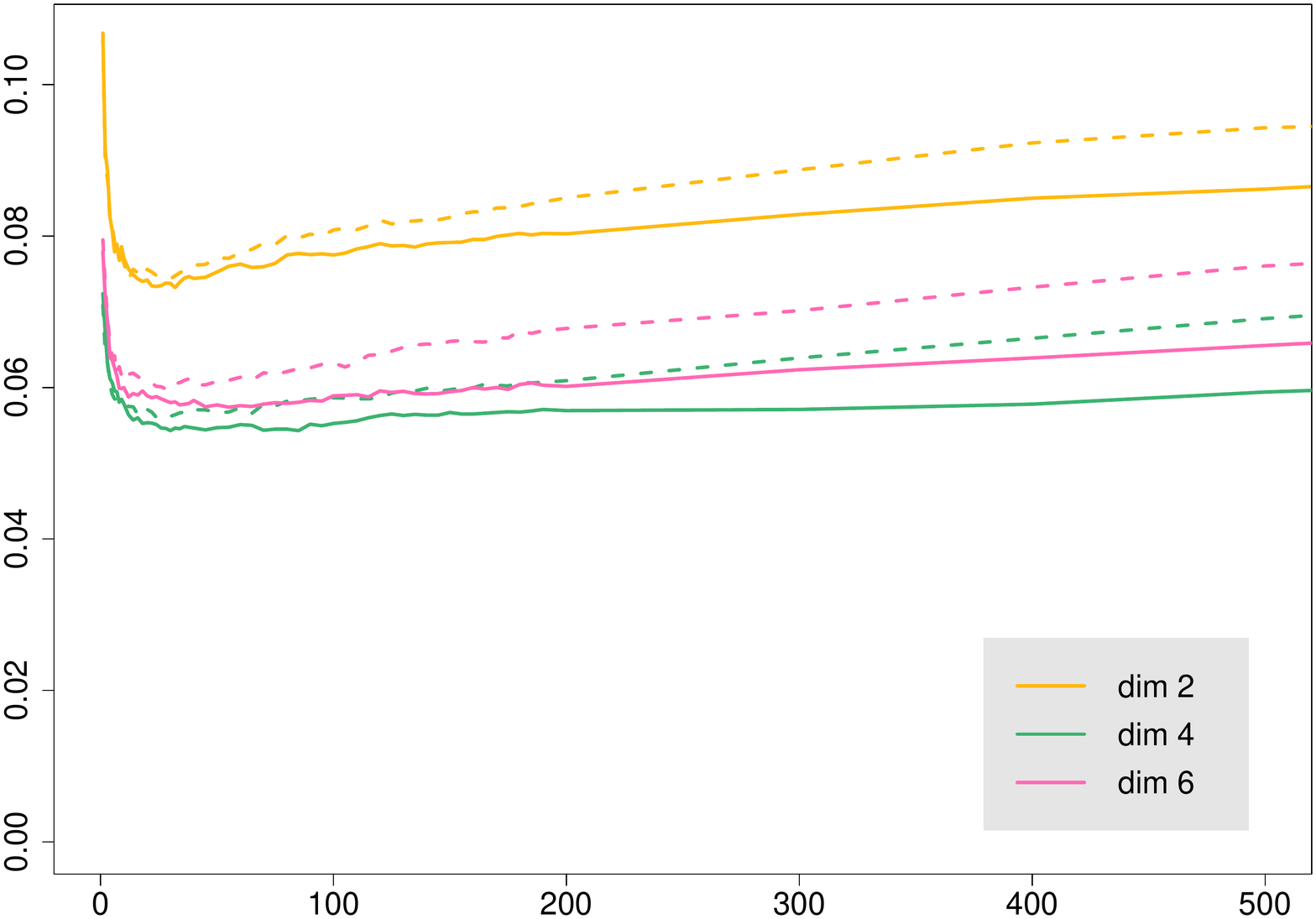} \hfill
    \includegraphics[width=.48\textwidth]{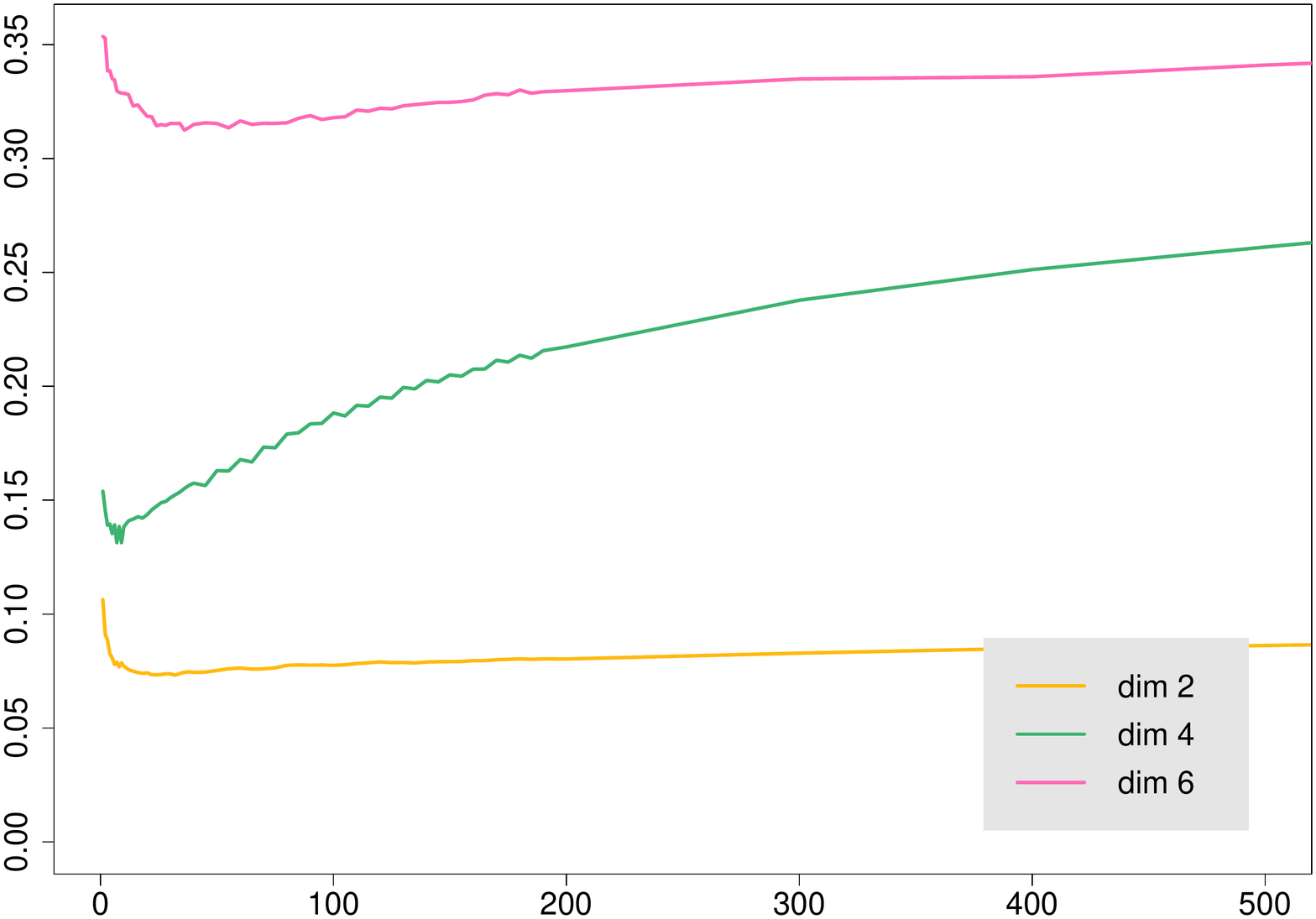}

    \hspace{.2\textwidth} (a) \hspace{.45\textwidth} (b) \hspace{.2\textwidth} ~\\

    \vspace{4mm}

    \includegraphics[width=.58\textwidth]{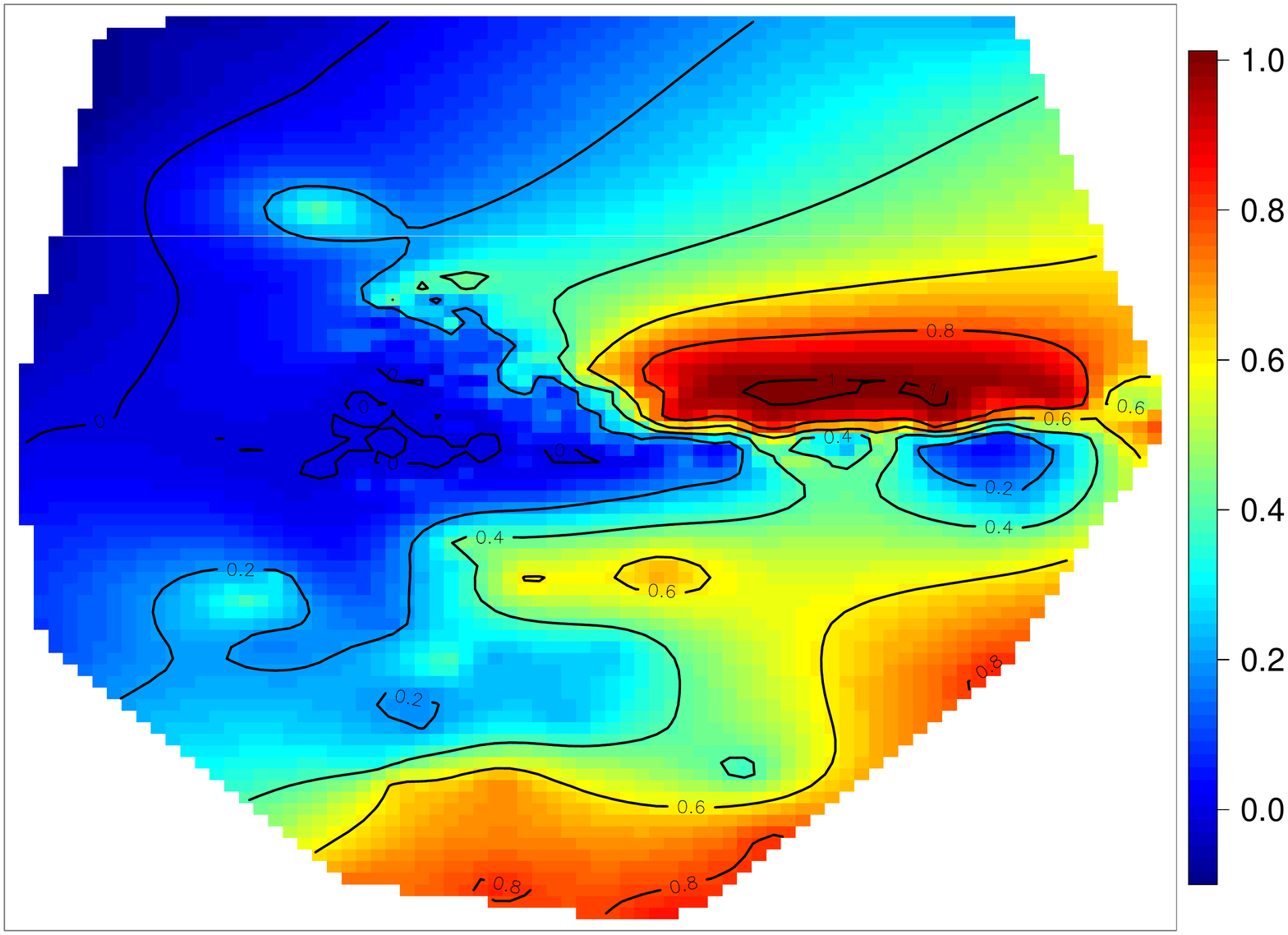}

    (c) 

   \caption{\textbf{First experiment results.} (a) Prior error rates
     (vertical axis) of ABC with respect to the number of nearest neighbors
     (horizontal axis) trained on a reference table of size $100,000$
     (solid lines) or $50,000$ (dashed lines), based on the 2D, 4D and
     6D summary statistics.  (b) Prior error rates of ABC based on the
     2D summary statistic compared with 4D and 6D summary statistics
     including additional ancillary statistics. (c) Evaluation of the
     local error on a 2D surface  }
  \label{fig:bicolor.discrete}
\end{figure*}

In all three experiments, we compare three nested sets of summary
statistics $S_{2D}(y)$, $S_{4D}(y)$ and $S_{6D}(y)$ of dimension 2, 4
and 6 respectively. They are defined as the projection onto the first
two (respectively four and six) axes of $S(y)$ described in the
previous section. We stress here that $S_{2D}(y)$, which is composed
of the summaries given by \citet{grelaud09}, are used beyond the
noise-free setting where they are sufficient for model choice. In
order to study the information carried by the connected components, we
add progressively our geometric summary statistics to the first set,
beginning by the $T(\mathscr G,y)$-type of statistics in $S_{4D}(y)$.
Finally, remark that, before evaluating the Euclidean distance in ABC
algorithms, we normalize the statistics in each reference tables with
respect to an estimation of their standard deviation since all these
summaries take values on axis of different scales. Simulated images
have been drawn thanks to the \citet{sw87} algorithm. In the least
favorable experiment, simulations of one hundred pictures (on pixel
grid of size $100\times 100$) via $20,000$ iterations of this
Markovian algorithm when parameters drawn from our prior requires
about one hour of computation on a single CPU with our optimized
\texttt{C++} code. Hence the amount of time required by ABC is
dominated by the simulations of $y$ via the Swedsen-Wang
algorithm. This motivated \citet*{moores14preproc} to propose a cut
down on the cost of running an ABC experiment by removing the
simulation of an image from hidden Potts model, and replacing it by an
approximate simulation of the summary statistics. Another alternative
is the clever sampler of  \citet{mira2001perfect} that provides exact
simulations of Ising models and can be extended to Potts models.

\paragraph{First experiment.} Fig.~\ref{fig:bicolor.discrete}(a)
illustrates the calibration of the number of nearest neighbors
(parameter $k$ of Algorithm~\ref{algo:ABC}) by showing the evolution
of the prior error rates (evaluated on a validation reference table
including $20,000$ simulations) when $k$ increases. We compared the
errors of six classifiers to inspect the differences between the three
sets of summary statistics (in yellow, green and magenta) and the
impact of the size of the training reference table ($100,000$
simulations in solid lines; $50,000$ simulations in dashed lines).
The numerical results exhibit that a good calibration of $k$ can
reduce the prior misclassification error. Thus, without really
degrading the performance of the classifiers, we can reduce the amount
of simulations required in the training reference table, whose
computation cost (in time) represents the main obstacle of ABC
methods, see also Table~\ref{tab:2colors}. Moreover, as can be guessed
from Fig.~\ref{fig:bicolor.discrete}(a), the sizes of the largest
connected components of induced graphs (included only in $S_{6D}(y)$)
do not carry additional information regarding the model choice and
Table~\ref{tab:2colors} confirms this results through evaluations of
the errors on a test reference table of $30,000$ simulations drawn
independently of both training and validation reference tables.

One can argue that the curse of dimensionality does not occur with
such low dimensional statistics and sizes of the training set, but
this intuition is wrong, as shown in
Fig.~\ref{fig:bicolor.discrete}(b).  The latter plot shows indeed the
prior misclassification rate as a function of $k$ when we replace the
last four summaries by ancillary statistics drawn independently of $m$
and $y$. We can conclude that, although the three sets of summary
statistics carry then the same information in this artificial setting,
the prior error rates increase substantially with the dimension
(classifiers are not trained on infinite reference tables!). This
conclusion shed new light on the results of
Fig.~\ref{fig:bicolor.discrete}(a): the $U(\mathscr G,y)$-type
summaries, based on the size of the largest component, are not
concretely able to help discriminate among models, but are either
highly correlated with the first four statistics; or the resolution (in
terms of size of the training reference table) does not permit the
exploitation of the possible information they add.

\begin{table}[t]
  \centering
  \caption{Evaluation of the prior error rate on a test reference
    table of size $20,000$ in the second experiment.}
  \label{tab:16colors}

  \textbf{Prior error rates}

  \begin{tabular}{ccc}
    \hline
    \textbf{Train size} & $\bf 50,000$ & $\bf 100,000$ 
    \\
    \hline
    2D statistics & $4.5\%$ & $4.4 \%$
    \\
    4D statistics & $4.6\%$ & $4.1 \%$
    \\
    6D statistics & $4.6\%$ & $4.3 \%$ \\
    \hline
  \end{tabular}
\end{table}

Fig.~\ref{fig:bicolor.discrete}(c) displays the local error rate with
respect to a projection of the image space on a plan. We have taken
here $S_1(y)=S_{2D}(y)$ in Definition~\ref{def:local}. And $S_2(y)$
ranges a plan given by a projection of the full set of
summaries that has been tuned empirically in order to gather the
errors committed by calls of $\widehat{m}(S_{2D}(y))$ on the
validation reference table. The most striking fact is that the local
error rises above $0.9$ in the oval, reddish area of
Fig.~\ref{fig:bicolor.discrete}(c). Other reddish areas of
Fig.~\ref{fig:bicolor.discrete}(c) in the bottom of the plot
correspond to parts of the space with very low probability, and may
be a dubious extrapolation of the Kriging algorithm. We can thus
conclude that the information of the new geometric summaries depends
highly on the position of $y$ in the image space and have
confidence in the interest of Algorithm~\ref{algo:adaptive.ABC}
(adaptive ABC) in this framework.  As exhibited in
Table~\ref{tab:2colors}(d), this last classifier does not decrease
dramatically the prior misclassification rates. But the errors of the
non-adaptive classifiers are already low and the error of any
classifier is bounded from below, as explained in
Proposition~\ref{pro:compare.knn.Bayes}. Interestingly though, the
adaptive classifier relies on $\widehat{m}(S_{2D}(y))$ (instead of the
most informative $\widehat{m}(S_{6D}(y))$) to take the final decision
at about $60 \%$ of the images of our test reference table of size
$30,000$.

\paragraph{Second experiment.} The framework was designed here to
study the limitations of our approach based on the connected
components of induced graphs. The number of latent colors is indeed
relatively high and the noise process do not rely on any ordering of
the colors to perturbate the pixels. Table~\ref{tab:16colors}
indicates the difficulty of capturing relevant information with the
geometric summaries we propose. Only the sharpness introduced by a
training reference table composed of $100,000$ simulations
distinguishes $\widehat{m}(S_{4D}(y))$ and $\widehat{m}(S_{6D}(y))$
from the basic classifier $\widehat{m}(S_{2D}(y))$. This conclusion is
reinforced by the low value of number of neighbors after the
calibration process, namely $k=16, 5$ and $5$ for
$\widehat{m}(S_{2D}(y))$, $\widehat{m}(S_{4D}(y))$ and
$\widehat{m}(S_{6D}(y))$ respectively. Hence we do not display in the
paper other diagnosis plots based on the prior error rates or the
conditional error rates, which led us to the same conclusion. The
adaptive ABC algorithm did not improve any of these results.

\begin{table}[b]
  \centering
  \caption{Evaluation of the prior error rate on a test reference
    table of size $30,000$ in the third experiment.}
  \label{tab:continuous}
  \textbf{Prior error rates}

  \begin{tabular}{cccc}
    \hline
    \textbf{Train size} & $\bf 5,000$ & $\bf 100,000$ 
    \\
    \hline
    2D statistics & $14.2 \%$  & $13.8 \%$
    \\
    4D statistics & $10.8 \%$ & $9.8 \%$
    \\
    6D statistics & $8.6 \%$  & $6.9 \%$ 
    \\
    Adaptive ABC & $8.2 \%$ & $6.7 \%$
    \\
    \hline
  \end{tabular}

\end{table}

\begin{figure*}
  \centering
    \includegraphics[width=.48\textwidth]{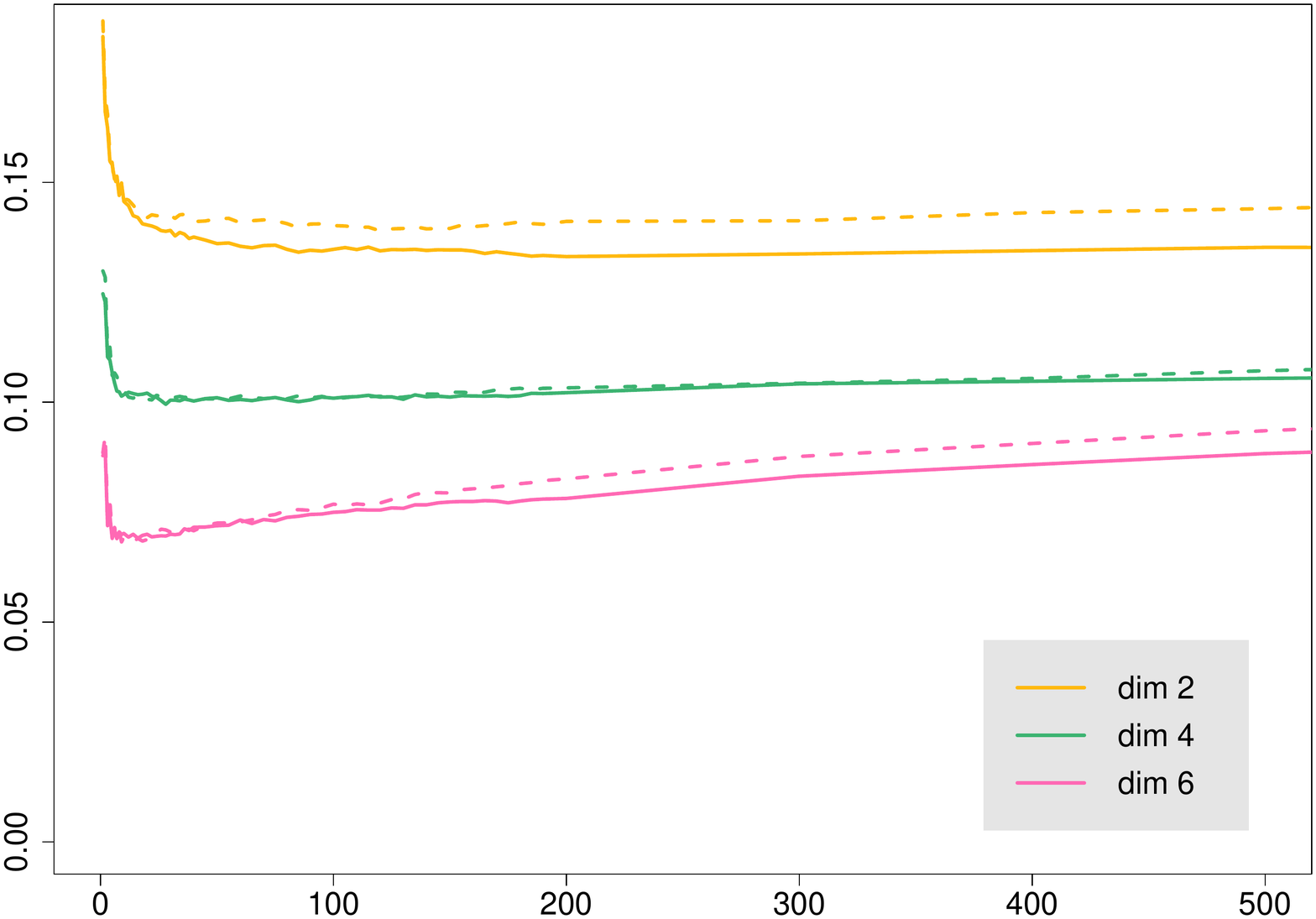} \hfill
    \includegraphics[width=.48\textwidth]{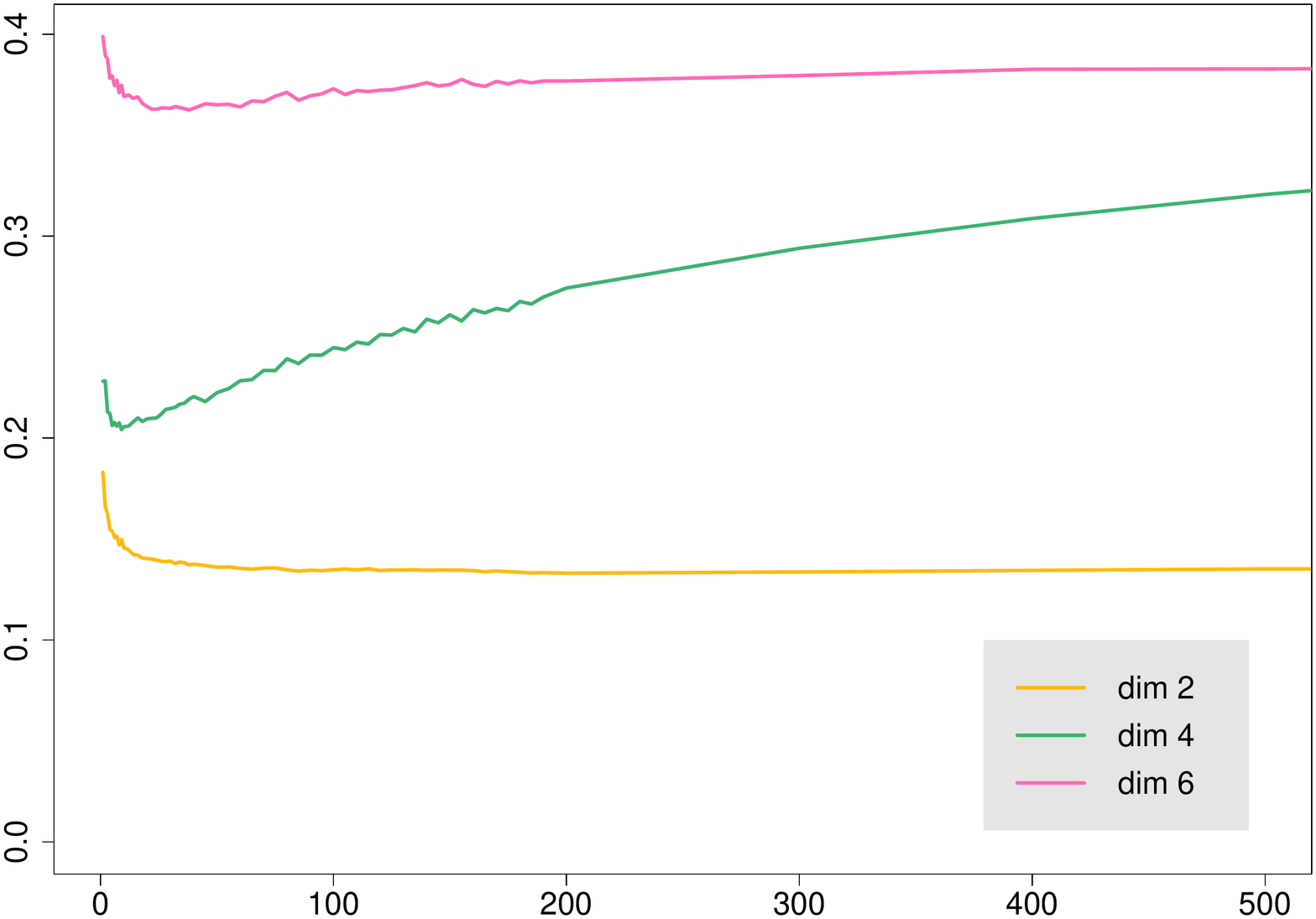}

    \hspace{.2\textwidth} (a) \hspace{.45\textwidth} (b) \hspace{.2\textwidth} ~

    \includegraphics[width=.48\textwidth]{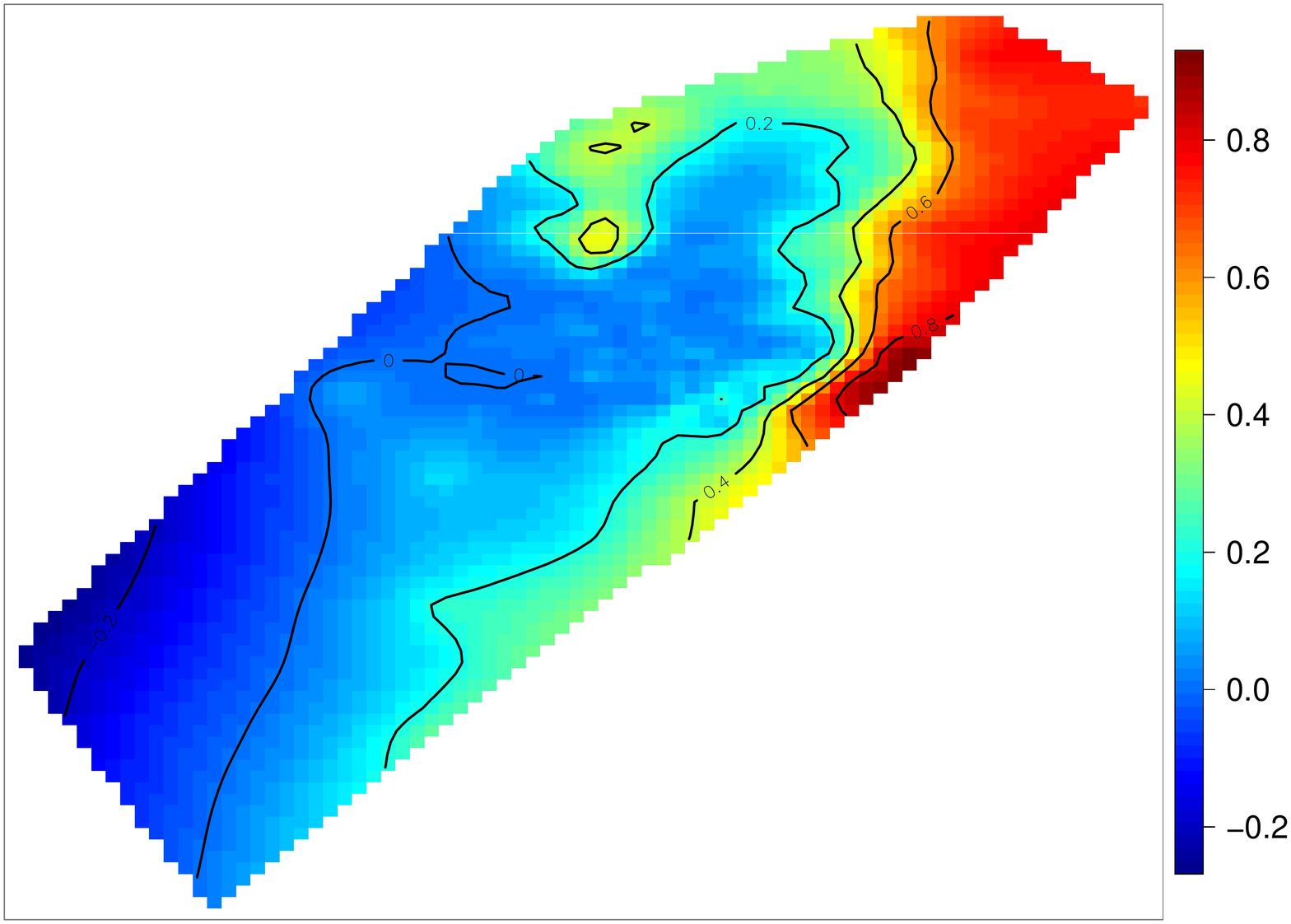}

    (c) 

   \caption{\textbf{Third experiment results.} (a) Prior error rates
     (vertical axis) of ABC with respect to the number of nearest neighbors
     (horizontal axis) trained on a reference table of size $100,000$
     (solid lines) or $50,000$ (dashed lines), based on the 2D, 4D and
     6D summary statistics.  (b) Prior error rates of ABC based on the
     2D summary statistics compared with 4D and 6D summary statistics
     including additional ancillary statistics. (c) Evaluation of the
     local error on a 2D surface }
  \label{fig:continuous}
\end{figure*}

\paragraph{Third experiment.} The framework here includes a continuous
noise process as described at the end of Section~\ref{sub:Potts}. We
reproduced the entire diagnosis process performed in the first
experiment and we obtained the results given in Fig.~\ref{fig:continuous} and
Table~\ref{tab:continuous}. The most noticeable difference is the
extra information carried by the $U(\mathscr G, y)$-statistics,
representing the size of the largest connected component, and the
adaptive ABC relie on the simplest $\widehat{m}(S_{2D}(y))$ in about
$30 \%$ of the data space (measured with the prior marginal
distribution in $y$). Likewise, the gain in misclassification errors is not
spectacular, albeit positive.


\section{Conclusion and perspective}

In the present article, we considered ABC model choice as a
classification problem in the framework of the Bayesian
paradigm (Section \ref{subsec-abc}) and provided a local error in order to
assess the accuracy of the classifier at $\yobs$ (Sections
\ref{sub:local} and \ref{sub:estimate.local.error}). We derived then
an adaptive classifier (Section \ref{sub:adaptive}) which is an
attempt to fight against the curse of dimensionality locally around
$\yobs$. This method contrasts with most projection methods which are
focused on parameter estimation \citep{blum2013}. Additionally, most
of them perform a global trade off between the dimension and the
information of the summary statistics over the whole prior domain,
while our proposal adapts the dimension with respect to $\yobs$
\citep[see also the discussion about the posterior loss approach
in][]{blum2013}.  Besides the inequalities of Proposition 
\ref{pro:compare.knn.Bayes} complement modestly the analysis of \citet{marin14} on ABC 
model choice. Principles of our proposal are well founded by avoiding the
well-known optimism of the training error rates and by resorting to
validation and test reference tables in order to evaluate the error
practically. And, finally, the machine learning viewpoint gives an
efficient way to calibrate the threshold of ABC (Section
\ref{subsec-abc}). 

Regarding latent Markov random fields, the proposed method of
constructing summary statistics based on the induced graphs (Section
\ref{sub:geometry}) yields a promising route to construct relevant
summary statistics in this framework. This approach is very intuitive
and can be reproduced in other settings. For instance, if the goal of
the Bayesian analysis is to select between isotropic latent Gibbs
models and anisotropic models, the averaged ratio between the width
and the length of the connect components or the ratio of the width and
the length of the largest connected components can be relevant
numerical summaries. We have also explained how to adapt the method to
a continuous noise by performing a quantization of the observed values
at each site of the fields (Section \ref{sub:geometry}). And the
detailed analysis of the numerical results demonstrates that the
approach is promising.  However the results on the $16$ color example
with a completely disordered noise indicate the limitation of the
induced graph approach. We believe that there exists a road we did not
explore above with an induced graph that add weights on the edges of
the graph according to the proximity of the colors, but the grouping
of sites on such weighted graph is not trivial.

The numerical results (Section \ref{sub:numerical}) highlighted that
the calibration of the number of neighbors in ABC provides better
results (in terms of misclassification) than a threshold set as a
fixed quantile of the distances between the simulated and the observed
datasets \citep[as proposed in][]{surveyABC}. Consequently, we can
reduce significantly the number of simulations in the reference table
without increasing the misclassification error rates. This represents
an important conclusion since the simulation of a latent Markov random
field requires a non-negligible amount of time.  The gain in
misclassification rates of the new summaries is real but not
spectacular and the adaptive ABC algorithm was able to select the most
performant classifier.


\section*{Acknowledgments}
The three author were financially supported by the Labex NUMEV.
We are grateful to Jean-Michel Marin for his constant feedback and
support. Some part of the present work was presented at MCMSki 4 in
January 2014 and benefited greatly from discussions with the
participants during the poster session. We would like to thank
the anonymous referrees and the Editors whose
valuable comments and insightful suggestions led to an improved
version of the paper. 

\bibliographystyle{abbrvnat}
\bibliography{biblio}

\begin{thebibliography}{37}
\providecommand{\natexlab}[1]{#1}
\providecommand{\url}[1]{\texttt{#1}}
\expandafter\ifx\csname urlstyle\endcsname\relax
  \providecommand{\doi}[1]{doi: #1}\else
  \providecommand{\doi}{doi: \begingroup \urlstyle{rm}\Url}\fi

\bibitem[Alf\`o et~al.(2008)Alf\`o, Nieddu, and Vicari]{alfo08}
M.~Alf\`o, L.~Nieddu, and D.~Vicari.
\newblock A finite mixture model for image segmentation.
\newblock \emph{Statistics and Computing}, 18\penalty0 (2):\penalty0 137--150,
  2008.

\bibitem[Baragatti and Pudlo(2014)]{baragatti14}
M.~Baragatti and P.~Pudlo.
\newblock An overview on {A}pproximate {B}ayesian {C}omputation.
\newblock \emph{ESAIM: Proc.}, 44:\penalty0 291--299, 2014.

\bibitem[Beaumont et~al.(2009)Beaumont, Cornuet, Marin, and
  Robert]{beaumont2009adaptive}
M.~A. Beaumont, J.-M. Cornuet, J.-M. Marin, and C.~P. Robert.
\newblock Adaptive approximate {B}ayesian computation.
\newblock \emph{Biometrika}, page asp052, 2009.

\bibitem[Besag(1974)]{besag74}
J.~Besag.
\newblock Spatial interaction and the statistical analysis of lattice systems
  (with {D}iscussion).
\newblock \emph{Journal of the Royal Statistical Society. Series B
  (Methodological)}, 36\penalty0 (2):\penalty0 192--236, 1974.

\bibitem[Besag(1975)]{besag75}
J.~Besag.
\newblock Statistical {A}nalysis of {N}on-{L}attice {D}ata.
\newblock \emph{The Statistician}, 24:\penalty0 179--195, 1975.

\bibitem[Biau et~al.(2013)Biau, C\'erou, and Guyader]{biau13}
G.~Biau, F.~C\'erou, and A.~Guyader.
\newblock New insights into {A}pproximate {B}ayesian {C}omputation.
\newblock \emph{Annales de l'Institut Henri Poincar\'e (B) Probabilités et
  Statistiques, in press}, 2013.

\bibitem[Blum et~al.(2013)Blum, Nunes, Prangle, and Sisson]{blum2013}
M.~G.~B. Blum, M.~A. Nunes, D.~Prangle, and S.~A. Sisson.
\newblock A {C}omparative {R}eview of {D}imension {R}eduction {M}ethods in
  {A}pproximate {B}ayesian {C}omputation.
\newblock \emph{Statistical Science}, 28\penalty0 (2):\penalty0 189--208, 2013.

\bibitem[Caimo and Friel(2011)]{caimo11}
A.~Caimo and N.~Friel.
\newblock Bayesian inference for exponential random graph models.
\newblock \emph{Social Networks}, 33\penalty0 (1):\penalty0 41--55, 2011.

\bibitem[Caimo and Friel(2013)]{caimo2013}
A.~Caimo and N.~Friel.
\newblock Bayesian model selection for exponential random graph models.
\newblock \emph{Social Networks}, 35\penalty0 (1):\penalty0 11 -- 24, 2013.

\bibitem[Cucala and Marin(2013)]{cucala13}
L.~Cucala and J.-M. Marin.
\newblock Bayesian {I}nference on a {M}ixture {M}odel {W}ith {S}patial
  {D}ependence.
\newblock \emph{Journal of Computational and Graphical Statistics}, 22\penalty0
  (3):\penalty0 584--597, 2013.

\bibitem[Del~Moral et~al.(2012)Del~Moral, Doucet, and
  Jasra]{delmoral:doucet:jasra:2009}
P.~Del~Moral, A.~Doucet, and A.~Jasra.
\newblock An adaptive sequential monte carlo method for approximate bayesian
  computation.
\newblock \emph{Statistics and Computing}, 22\penalty0 (5):\penalty0
  1009--1020, 2012.

\bibitem[Devroye et~al.(1996)Devroye, Gy{\"o}rfi, and
  Lugosi]{devroye:gyorfi:lugosi:1996}
L.~Devroye, L.~Gy{\"o}rfi, and G.~Lugosi.
\newblock \emph{A probabilistic theory of pattern recognition}, volume~31 of
  \emph{Applications of Mathematics (New York)}.
\newblock Springer-Verlag, New York, 1996.

\bibitem[Didelot et~al.(2011)Didelot, Everitt, Johansen, and
  Lawson]{didelot2011}
X.~Didelot, R.~G. Everitt, A.~M. Johansen, and D.~J. Lawson.
\newblock Likelihood-free estimation of model evidence.
\newblock \emph{Bayesian Analysis}, 6\penalty0 (1):\penalty0 49--76, 2011.

\bibitem[Druilhet and Marin(2007)]{druilhet:marin:2007}
P.~Druilhet and J.-M. Marin.
\newblock Invariant {HPD} credible sets and {MAP} estimators.
\newblock \emph{Bayesian Analysis}, 2\penalty0 (4):\penalty0 681--691, 2007.

\bibitem[Estoup et~al.(2012)Estoup, Lombaert, Marin, Robert, Guillemaud, Pudlo,
  and Cornuet]{estoup:etal:2012}
A.~Estoup, E.~Lombaert, J.-M. Marin, C.~Robert, T.~Guillemaud, P.~Pudlo, and
  J.-M. Cornuet.
\newblock Estimation of demo-genetic model probabilities with {A}pproximate
  {B}ayesian {C}omputation using linear discriminant analysis on summary
  statistics.
\newblock \emph{Molecular Ecology Ressources}, 12\penalty0 (5):\penalty0
  846--855, 2012.

\bibitem[Everitt(2012)]{everitt12}
R.~G. Everitt.
\newblock Bayesian {P}arameter {E}stimation for {L}atent {M}arkov {R}andom
  {F}ields and {S}ocial {N}etworks.
\newblock \emph{Journal of Computational and Graphical Statistics}, 21\penalty0
  (4):\penalty0 940--960, 2012.

\bibitem[Forbes and Peyrard(2003)]{forbes03}
F.~Forbes and N.~Peyrard.
\newblock Hidden {M}arkov random field model selection criteria based on mean
  field-like approximations.
\newblock \emph{Pattern Analysis and Machine Intelligence, IEEE Transactions
  on}, 25\penalty0 (9):\penalty0 1089--1101, 2003.

\bibitem[Fran\c{c}ois et~al.(2006)Fran\c{c}ois, Ancelet, and
  Guillot]{francois06}
O.~Fran\c{c}ois, S.~Ancelet, and G.~Guillot.
\newblock Bayesian {C}lustering {U}sing {H}idden {M}arkov {R}andom {F}ields in
  {S}patial {P}opulation {G}enetics.
\newblock \emph{Genetics}, 174\penalty0 (2):\penalty0 805--816, 2006.

\bibitem[Friel(2012)]{frielproc}
N.~Friel.
\newblock Bayesian inference for {G}ibbs random fields using composite
  likelihoods.
\newblock In \emph{Simulation Conference (WSC), Proceedings of the 2012
  Winter}, pages 1--8, 2012.

\bibitem[Friel(2013)]{frielevidence}
N.~Friel.
\newblock Evidence and {B}ayes {F}actor {E}stimation for {G}ibbs {R}andom
  {F}ields.
\newblock \emph{Journal of Computational and Graphical Statistics}, 22\penalty0
  (3):\penalty0 518--532, 2013.

\bibitem[Friel and Rue(2007)]{friel07}
N.~Friel and H.~Rue.
\newblock Recursive computing and simulation-free inference for general
  factorizable models.
\newblock \emph{Biometrika}, 94\penalty0 (3):\penalty0 661--672, 2007.

\bibitem[Friel et~al.(2009)Friel, Pettitt, Reeves, and Wit]{frielcache}
N.~Friel, A.~N. Pettitt, R.~Reeves, and E.~Wit.
\newblock Bayesian {I}nference in {H}idden {M}arkov {R}andom {F}ields for
  {B}inary {D}ata {D}efined on {L}arge {L}attices.
\newblock \emph{Journal of Computational and Graphical Statistics}, 18\penalty0
  (2):\penalty0 243--261, 2009.

\bibitem[Green and Richardson(2002)]{green02}
P.~J. Green and S.~Richardson.
\newblock Hidden {M}arkov {M}odels and {D}isease {M}apping.
\newblock \emph{Journal of the {A}merican {S}tatistical {A}ssociation},
  97\penalty0 (460):\penalty0 1055--1070, 2002.

\bibitem[Grelaud et~al.(2009)Grelaud, Robert, Marin, Rodolphe, and
  Taly]{grelaud09}
A.~Grelaud, C.~P. Robert, J.-M. Marin, F.~Rodolphe, and J.-F. Taly.
\newblock {ABC} likelihood-free methods for model choice in {G}ibbs random
  fields.
\newblock \emph{Bayesian Analysis}, 4\penalty0 (2):\penalty0 317--336, 2009.

\bibitem[Hurn et~al.(2003)Hurn, Husby, and Rue]{hurn03}
M.~A. Hurn, O.~K. Husby, and H.~Rue.
\newblock A {T}utorial on {I}mage {A}nalysis.
\newblock In \emph{Spatial Statistics and Computational Methods}, volume 173 of
  \emph{Lecture Notes in Statistics}, pages 87--141. Springer New York, 2003.
\newblock ISBN 978-0-387-00136-4.

\bibitem[Marin et~al.(2012)Marin, Pudlo, Robert, and Ryder]{surveyABC}
J.-M. Marin, P.~Pudlo, C.~P. Robert, and R.~J. Ryder.
\newblock {A}pproximate {B}ayesian {C}omputational methods.
\newblock \emph{Statistics and Computing}, 22\penalty0 (6):\penalty0
  1167--1180, 2012.

\bibitem[Marin et~al.(2013)Marin, Pillai, Robert, and Rousseau]{marin14}
J.-M. Marin, N.~S. Pillai, C.~P. Robert, and J.~Rousseau.
\newblock Relevant statistics for {B}ayesian model choice.
\newblock \emph{Journal of the {R}oyal {S}tatistical {S}ociety: {S}eries {B}
  ({S}tatistical {M}ethodology)}, 2013.

\bibitem[Marjoram et~al.(2003)Marjoram, Molitor, Plagnol, and
  Tavar{\'e}]{marjoram:etal:2003}
P.~Marjoram, J.~Molitor, V.~Plagnol, and S.~Tavar{\'e}.
\newblock Markov chain {M}onte {C}arlo without likelihoods.
\newblock \emph{Proceedings of the National Academy of Sciences}, 100\penalty0
  (26):\penalty0 15324--15328, 2003.

\bibitem[Mira et~al.(2001)Mira, M{\o}ller, and Roberts]{mira2001perfect}
A.~Mira, J.~M{\o}ller, and G.~O. Roberts.
\newblock Perfect slice samplers.
\newblock \emph{Journal of the Royal Statistical Society: Series B (Statistical
  Methodology)}, 63\penalty0 (3):\penalty0 593--606, 2001.

\bibitem[Moores et~al.(2014)Moores, Hargrave, Harden, and Mengersen]{moores14}
M.~T. Moores, C.~E. Hargrave, F.~Harden, and K.~Mengersen.
\newblock Segmentation of cone-beam {CT} using a hidden {M}arkov random field
  with informative priors.
\newblock \emph{Journal of Physics : Conference Series}, 489, 2014.

\bibitem[{Moores} et~al.(2014){Moores}, {Mengersen}, and
  {Robert}]{moores14preproc}
M.~T. {Moores}, K.~{Mengersen}, and C.~P. {Robert}.
\newblock Pre-processing for approximate {B}ayesian computation in image
  analysis.
\newblock \emph{ArXiv e-prints}, March 2014.

\bibitem[Prangle et~al.(2013)Prangle, Fearnhead, Cox, Biggs, and
  French]{prangle14}
D.~Prangle, P.~Fearnhead, M.~P. Cox, P.~J. Biggs, and N.~P. French.
\newblock Semi-automatic selection of summary statistics for {ABC} model
  choice.
\newblock \emph{Statistical {A}pplications in {G}enetics and {M}olecular
  {B}iology}, pages 1--16, 2013.

\bibitem[Pritchard et~al.(1999)Pritchard, Seielstad, Perez-Lezaun, and
  Feldman]{pritchard99}
J.~K. Pritchard, M.~T. Seielstad, A.~Perez-Lezaun, and M.~W. Feldman.
\newblock Population growth of human {Y} chromosomes: a study of {Y} chromosome
  microsatellites.
\newblock \emph{Molecular Biology and Evolution}, 16\penalty0 (12):\penalty0
  1791--1798, 1999.

\bibitem[Reeves and Pettitt(2004)]{reeves04}
R.~Reeves and A.~N. Pettitt.
\newblock Efficient recursions for general factorisable models.
\newblock \emph{Biometrika}, 91\penalty0 (3):\penalty0 751--757, 2004.

\bibitem[Robert et~al.(2011)Robert, Cornuet, Marin, and Pillai]{robert11}
C.~P. Robert, J.-M. Cornuet, J.-M. Marin, and N.~S. Pillai.
\newblock Lack of confidence in approximate {B}ayesian computation model
  choice.
\newblock \emph{Proceedings of the National Academy of Sciences}, 108\penalty0
  (37):\penalty0 15112--15117, 2011.

\bibitem[Swendsen and Wang(1987)]{sw87}
R.~H. Swendsen and J.-S. Wang.
\newblock Nonuniversal critical dynamics in {M}onte {C}arlo simulations.
\newblock \emph{Physical {R}eview {L}etters}, 58\penalty0 (2):\penalty0 86--88,
  1987.

\bibitem[Tavar{\'e} et~al.(1997)Tavar{\'e}, Balding, Griffiths, and
  Donnelly]{tavare97}
S.~Tavar{\'e}, D.~J. Balding, R.~C. Griffiths, and P.~Donnelly.
\newblock Inferring {C}oalescence {T}imes {F}rom {DNA} {S}equence {D}ata.
\newblock \emph{Genetics}, 145\penalty0 (2):\penalty0 505--518, 1997.

\end{thebibliography}
\end{document}